\documentclass{article}

\usepackage{graphicx}
\usepackage{xcolor}
\usepackage{hyperref}
\usepackage{dsfont}

\usepackage{tikz-cd} 
\usepackage{enumerate} 

\usepackage{pgf,pgfplots}
\pgfplotsset{compat=1.17}
\usepackage{tikz,tikz-3dplot}
\usetikzlibrary{math}

\usepackage[margin=1in, top=0.5in]{geometry}

\usepackage{amsmath,amsfonts,amssymb,amsthm,mathrsfs}

\usepackage{cleveref}

\usepackage{subcaption}

\theoremstyle{plain} 

\newtheorem{theorem}{Theorem}

\newtheorem{proposition}[theorem]{Proposition}
\newtheorem{corollary}[theorem]{Corollary}

\theoremstyle{definition} 

\newtheorem{definition}[theorem]{Definition}

\newtheorem{example}[theorem]{Example}

\newcommand{\R}{\mathbb{R}}
\newcommand{\Z}{\mathbb{Z}}

\newcommand{\Fi}{\mathbb{F}}

\newcommand{\B}{\mathcal{B}}

\newcommand{\norm}[1]{\left\|#1\right\|}

\DeclareMathOperator{\Span}{Span}

\newcommand\TODO[3]{\hbox to 0pt{\textcolor{#1}{$^\bullet$}}\marginpar{\footnotesize \textcolor{#1}{\begin{flushleft}#2: #3\end{flushleft}}}}

\title{About non-uniqueness when removing closed orbits in Morse-Smale vector fields}
\author{Clemens Bannwart\footnote{Department of Physics, Mathematics and Computer Science, University of Modena and Reggio Emilia, Italy,
\texttt{clemens.bannwart@gmail.com}}}
\date{}

\begin{document}

\maketitle

\begin{abstract}
    Given a Morse-Smale vector field on a smooth manifold, Franks \cite{Franks_MSflows_1979} described how one can replace a closed orbit of index $k$ by two rest points of index $k+1$ and $k$, using a local perturbation. Combined with classical results about gradient-like vector fields, this gives a method of assigning different topological or algebraic structures to Morse-Smale vector fields.
    We show that there are multiple non-equivalent ways of following this procedure and illustrate this non-uniqueness in various examples. We describe the consequences of this non-uniqueness to the endeavour of assigning CW complexes or chain complexes to Morse-Smale vector fields in a canonical way, contrary to what is stated in \cite[Corollary 5.2]{Franks_MSflows_1979} and \cite[Theorem 2.5]{EidiJost2022}, respectively.
\end{abstract}

\section{Why and how to remove closed orbits}

Assume that $M$ is a closed smooth manifold and $X$ is a Morse-Smale vector field on $M$. A lot of information can be gained by studying the critical elements of $X$, i.e. the rest points and closed orbits. An index can be assigned to each critical element, describing the local behaviour of $X$. Generalized Morse inequalities were proven by Smale, relating the number of critical elements of each index of $X$ to the Betti numbers of $M$~\cite{Smale1960MorseIineq}. A special case that is better understood is when $X$ is gradient-like, i.e. has no closed orbits. In that case it was pointed out in \cite{Franks_MSflows_1979}, that by following \cite{Milnor1963MorseTheory} we can construct a CW complex $Y$ whose $k$-cells are in one-to-one correspondence with the rest points of index $k$ of $X$. This CW complex is defined uniquely up to cell equivalence and there exists a homotopy equivalence $M\to Y$ that maps the unstable manifold of every rest point into the closure of the corresponding cell. Another structure that can be defined in terms of the critical elements of $X$ is the Morse complex (also sometimes denoted by some combination of the names Morse, Smale, Witten, Thom, Floer, see \cite{Bott_Indomitable88} for a historical discussion). This is a chain complex with coefficients in some field $\Fi$ (often $\Fi=\Z_2$ is chosen for easier computations), such that for every $k$, the vector space in degree $k$ has a basis corresponding to the rest points of index $k$ of $X$. The differential of this chain complex is defined by counting certain flow lines and its homology agrees with the singular homology of $M$, see e.g. \cite{BanyagaLecturesOnMorseHomology} for details.

The question arises whether we can define analogous structures more generally for Morse-Smale vector fields also in the presence of closed orbits. An indicator for this is the following result by Franks, allowing us to replace a closed orbit of index $k$ by a pair of rest points of index $k+1$ and $k$, respectively.

\begin{proposition}[{\cite[Proposition 5.1]{Franks_MSflows_1979}}] \label{prop:replace-closed-orbit}
    Suppose $X$ is a Morse-Smale vector field on an orientable manifold whose flow has a closed orbit $\gamma$ of index $k$ in standard form. Then given a small neighborhood $U$ of $\gamma$, there exists a Morse-Smale vector field $X'$ which agrees with $X$ outside $U$ and in $U$ has rest points $p$, $q$ of index $k+1$ and $k$ but no other rest points or closed orbits. Also, $q$ is a successor to $p$ and their connecting manifold consists of two points framed with opposite orientations if the original closed orbit $\gamma$ was untwisted and otherwise framed with the same orientation. Moreover the unstable manifold for $\gamma$ will be equal to $W^u(p) \cup W^u(q)$.
\end{proposition}

\begin{figure}[t]
    \centering
    \includegraphics[width=0.35\linewidth]{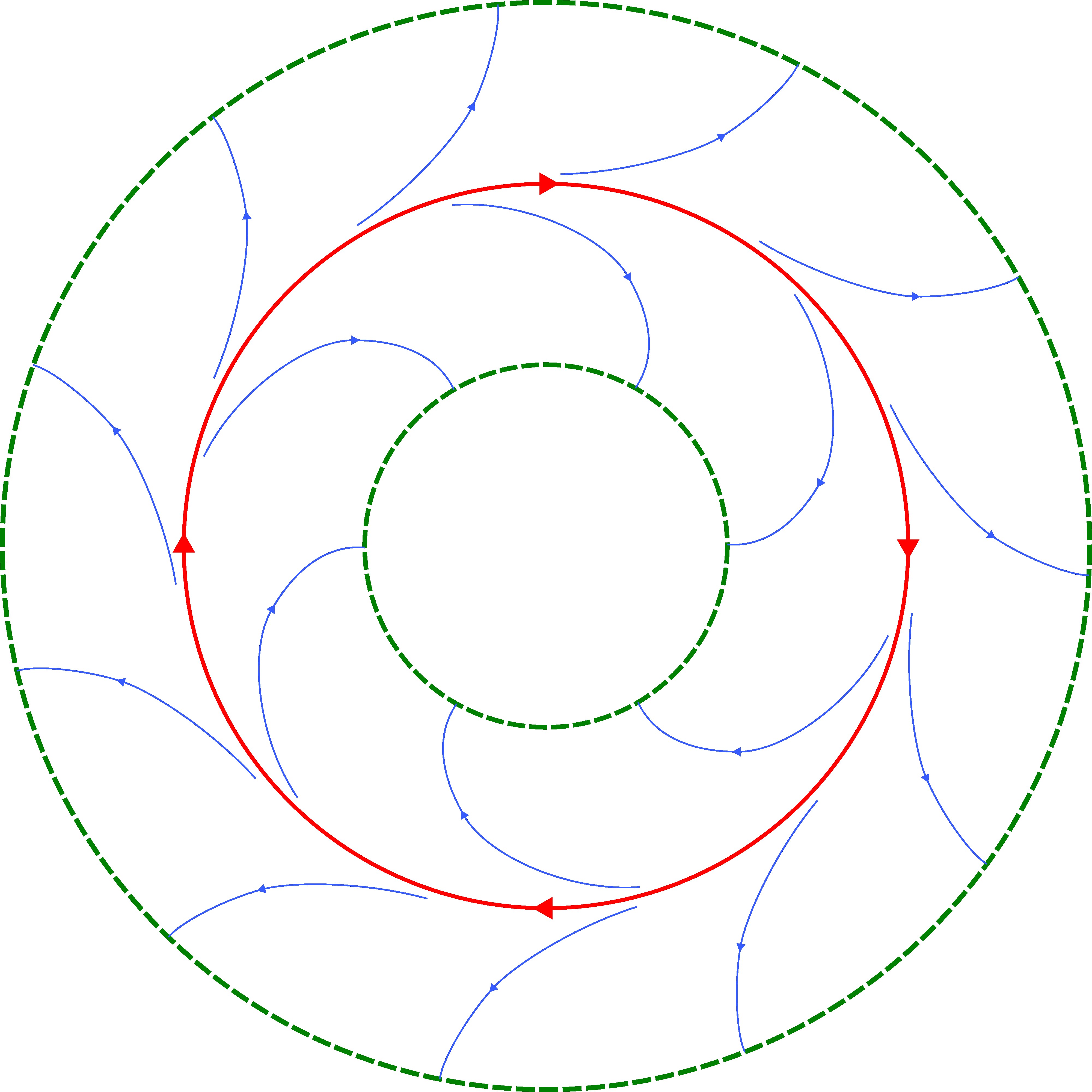}
    \hspace{0.1\linewidth}
    \includegraphics[width=0.35\linewidth]{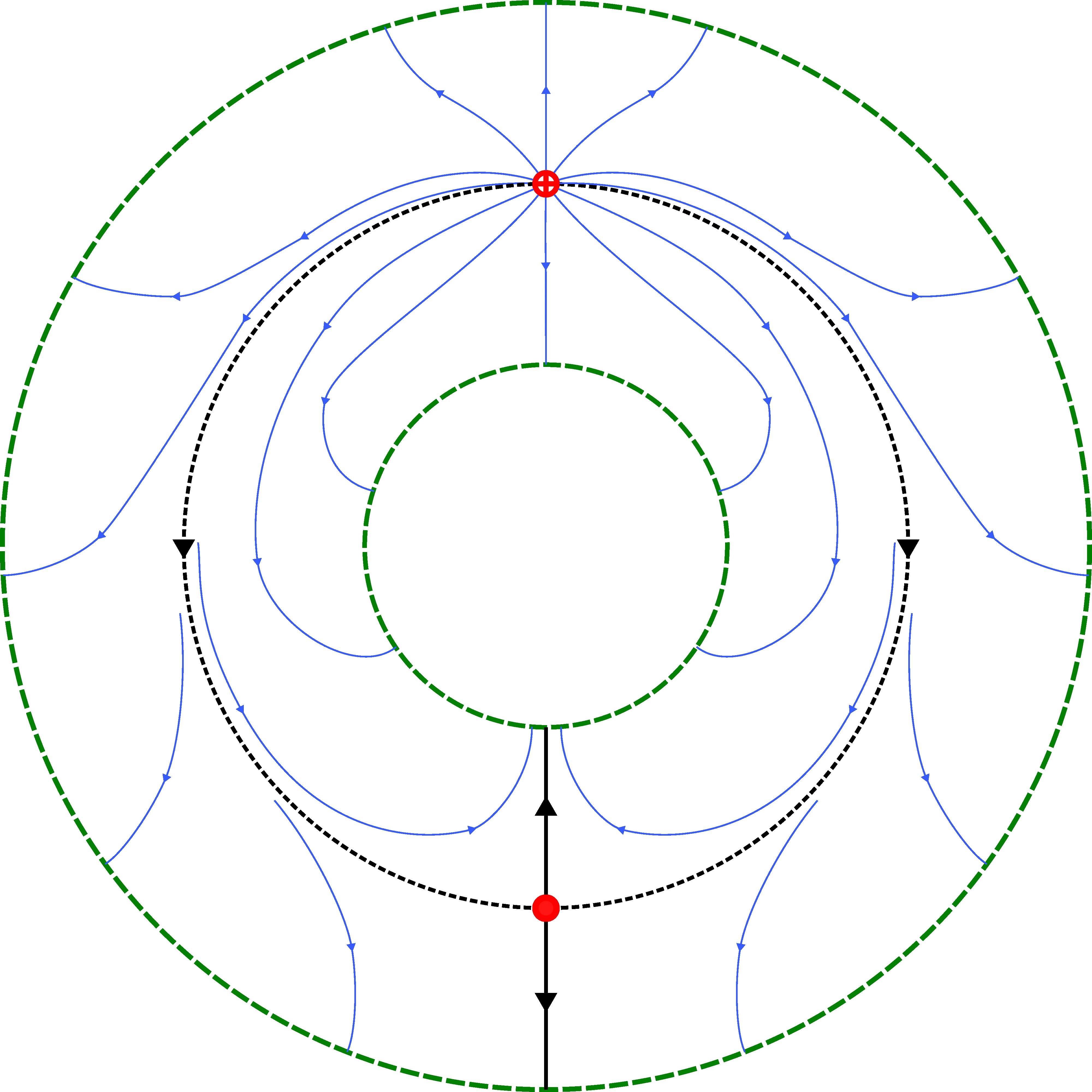}
    \caption{By \Cref{prop:replace-closed-orbit}, we can replace a closed orbit of index $k$ (left panel) by two rest points of index $k+1$ and $k$ (right panel). The vector field stays unchanged outside of a neighbourhood of the orbit, which can be chosen arbitrarily small.}
    \label{fig:perturb}
\end{figure}

Roughly speaking, this works by reversing the direction of $X$ on one half of the orbit, thereby creating two rest points that lie on opposite ends of the original closed orbit. The other points of what before was a closed orbit now form two flow lines from the rest point of index $k+1$ to the rest point of index $k$, see \Cref{fig:perturb}. By repeating this for every closed orbit, we can replace $X$ by a gradient-like vector field $X'$. 

This yields the following pipeline: Start with a Morse-Smale vector field $X$. Replace all closed orbits by pairs of rest points according to \Cref{prop:replace-closed-orbit} to obtain a gradient-like Morse-Smale vector field $X'$. Use the known results to assign some structure to $X'$ (e.g. a CW complex or a chain complex). Consider the choices that were made when perturbing $X$ to $X'$ and check that the resulting structure is unique (up to some equivalence relation).

We first examine the uniqueness question for CW complexes assigned to Morse-Smale vector fields. It was claimed in \cite[Corollary 5.2]{Franks_MSflows_1979} that this can be used to assign a CW complex to $X$ which is unique up to cell equivalence. In \Cref{sec:CW} we review the definition of cell equivalence and some relevant results, then provide examples that illustrate that the CW complex resulting from this procedure may not be unique up to cell equivalence, showing that the uniqueness part in \cite[Corollary 5.2]{Franks_MSflows_1979} is incorrect.

Next we examine the uniqueness question for assigning chain complexes to Morse-Smale vector fields. 
In \Cref{sec:chain} we consider the vector spaces and linear maps defined in \cite{EidiJost2022}. We show that in dimension two, these maps square to zero, but the resulting homology may differ from the singular homology of the underlying manifold. Then we present an example in dimension three where these maps do not square to zero, thereby showing that \cite[Theorem 2.5]{EidiJost2022} is incorrect.

Let us make some remarks about how we display our examples. We draw rest points and closed orbits in red. For rest points, we use different symbols, according to the index. Rest points of index $0$ (sinks) are drawn as a minus sign. Rest points of index $1$ (saddles) are drawn as a filled dot. Rest points of index $2$ (sources) are drawn as a plus sign. In \Cref{subfig:MSvf-R3} we have a three-dimensional vector field, where we draw the rest points of index $3$ as a hollow dot. \Cref{fig:closed-orbit-1,fig:perturbation1,fig:perturbation2} display vector fields on the $2$-sphere, which we visualize as a disk, the green boundary representing a single point of the sphere. This point is a rest point of index $0$ of the vector field in each case. Flow lines flowing to and from saddle points are drawn in black. Other flow lines are drawn in blue, with thinner lines. In some cases, these additional flow lines are not drawn, as their exact shape is not relevant for our purposes.

\section{Non-uniqueness for CW complexes}\label{sec:CW}

We recall the definition of cell equivalence, since it is a rather specific notion of \cite{Franks_MSflows_1979}. We start by introducing the face poset.
If $e$ and $e'$ are cells of a CW complex $Y$, we will say that $e' \le e$ if the closure of $e$ contains any part of the interior of $e'$. If we make this relation transitive (and denote the new transitive relation by $\le$ as well) then we obtain a partial ordering on the cells of $Y$. If $S$ is a subset of the cells of $Y$ with the property that $e \in S$ and $e'\le e$ implies $e' \in S$, then the union of the cells in $S$ forms a subcomplex of $Y$. In particular, if $e$ is a cell of $Y$ we define the \textbf{base} of $e$, denoted $Y(e)$, to be the smallest subcomplex of $Y$ containing $e$. Thus $Y(e)$ is the union of all cells $e'$ in $Y$ such that $e'\le e$.

\begin{definition}[{\cite[Definition 2.1]{Franks_MSflows_1979}}]\label{def:cell-equivalence-original}
    Two finite CW complexes $Y$ and $Y'$ will be called \textbf{cell equivalent} providing there is a homotopy equivalence $h\colon Y \to Y'$ with the property that there is a one-to-one correspondence between cells of $Y$ and cells of $Y'$ such that if $e\subseteq Y$ corresponds to $e'\subseteq Y'$ then $h$ maps $Y(e)$ to $Y'(e')$ and is a homotopy equivalence of these subcomplexes.
\end{definition}

There is a problem with this definition of cell equivalance, namely that it does not yield an equivalence relation. This is showcased by \Cref{fig:cell-equiv-not-sym}, where we have two CW complexes, $Y,Y'$, such that a cell equivalence $h\colon Y\to Y'$ exists, but not in the other direction. We thus present an alternative version of the definition of cell equivalence, where we demand the existence of a homotopy inverse that is also a cell equivalence and that the base of any cell gets mapped surjectively onto the base of its corresponding cell.

\begin{figure}[t]
    \centering
    \begin{subfigure}{0.4\linewidth}
      \includegraphics[height=3cm]{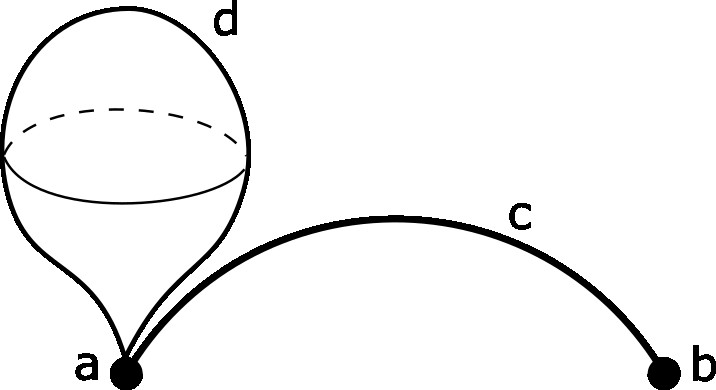}
      \caption{The CW complex $Y$ consists of the 0-cells $a$ and $b$, the 1-cell $c$, and the 2-cell $d$, which is attached to the 0-cell $a$.}
    \end{subfigure}
    \hspace{0.15\linewidth}
    \begin{subfigure}{0.4\linewidth}
      \includegraphics[height=3.2cm]{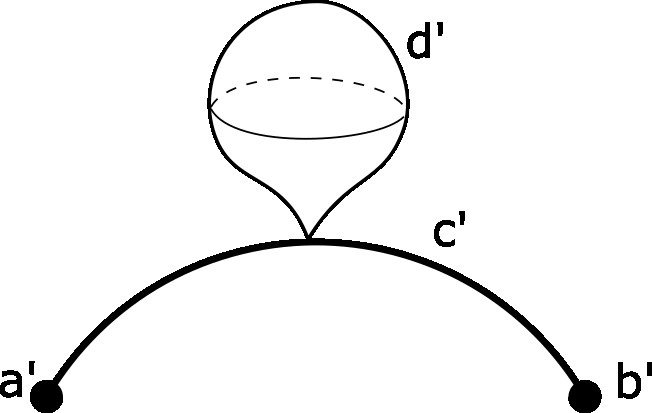}
      \caption{The CW complex $Y'$ consists of the 0-cells $a'$ and $b'$, the 1-cell $c'$, and the 2-cell $d'$, which is attached in the middle of the 1-cell $c'$.}
    \end{subfigure}
    \caption{In $Y$, the base of the cell $d$ consists only of the cells $d$ and $a$. In $Y'$, the base of the cell $d'$ is all of $Y'$. The map $h\colon Y \to Y'$ that maps $a\mapsto a'$, $b\mapsto b'$, $c\mapsto c'$, the bottom half of $d$ to the left half of $c'$, and the upper half of $d$ to $d'$, is a cell equivalence according to \Cref{def:cell-equivalence-original}. However, no cell equivalence exists in the other direction, so it does not satisfy the stronger \Cref{def:cell-equivalence-symmetric}.}
    \label{fig:cell-equiv-not-sym}
\end{figure}

\begin{definition}\label{def:cell-equivalence-symmetric}
    Two finite CW complexes $Y$ and $Y'$ are called \textbf{cell equivalent} if there are maps $f\colon Y \to Y'$ and $g\colon Y' \to Y$ that are homotopy inverses to each other, such that there exists a one-to-one correspondence between the cells of $Y$ and the cells of $Y'$ such that if $e\subset Y$ corresponds to $e'\subset Y$, then $f(Y(e)) = Y'(e')$, $g(Y'(e')) = Y(e)$, and the restrictions of $f$ to $Y(e)$ and of $g$ to $Y'(e')$ are homotopy inverses to each other.
\end{definition}

In this version, symmetry is already built into the definition, so it becomes straightforward to check that it is indeed an equivalence relation on finite CW complexes. Moreover, if $Y$ and $Y'$ are cell equivalent in the sense of \Cref{def:cell-equivalence-symmetric}, it follows that their face posets are isomorphic. 

Given a closed smooth manifold $M$ and a critical element $\beta$ of a vector field on $M$, we denote by $W^s(\beta)$ and $W^u(\beta)$ its stable and unstable manifolds. See \cite{Franks_MSflows_1979} for the necessary definitions and background. 
The following statement is a reformulation of a fundamental result from Morse theory, stating that we can assign to each gradient-like Morse-Smale vector field a CW complex $Y$ unique up to cell equivalence, with the homotopy type of the underlying manifold. 

\begin{theorem}[{\cite[Theorem 2.3]{Franks_MSflows_1979}}]\label{thm:morse-theory-main}
    If $X$ is a gradient-like Morse-Smale vector field on $M$, then there exists a CW complex $Y$, unique up to cell equivalence, and a homotopy equivalence $g\colon M \to Y$ such that for each rest point $p$ of index $k$, $g(W^u(p))$ is contained in the base $Y(e)$ of a single $k$-cell $e$.

    In this way $g$ establishes a one-to-one correspondence between rest points of $X$ of index $k$ and $k$-cells of $Y$. Moreover, the partial order $<$ on rest points defined by $q \le p$ if and only if $W^s(q) \cap W^u(p) \neq \emptyset$ corresponds to the partial order $\le$ on the cells of $Y$.
\end{theorem}

Combining \Cref{thm:morse-theory-main} with \Cref{prop:replace-closed-orbit}, we get the following result, stating that we can do the same thing also for general Morse-Smale vector fields. The underlined part about uniqueness is not true, as it is showcased next by \Cref{ex:minimal-example,ex:not-cell-equiv}. If we delete that part the statement becomes correct. The author of \cite{Franks_MSflows_1979} agrees with this assessment \cite{FranksLetter}.

\begin{corollary}[{\cite[Corollary 5.2]{Franks_MSflows_1979}}] \label{cor:ms-field-cw-complex}
    For any Morse-Smale vector field $X$ on a compact manifold $M$ there is a CW complex $Y$ homotopy equivalent to $M$ \underline{and unique up to cell equivalence} associated to $X$. Corresponding to each rest point of $X$ of index $k$ there is a $k$-cell of $Y$, and to each closed orbit of $X$ of index $k$ there are cells of $Y$ of dimension $k+1$ and $k$ attached by a map of degree $0$ or $2$ depending on whether or not the orbit is twisted.
\end{corollary}

The proof of \Cref{cor:ms-field-cw-complex} (without the underlined part) goes as follows: Given a Morse-Smale vector field $X$, apply \Cref{prop:replace-closed-orbit} to each closed orbit in order to obtain a gradient-like vector field $X'$. Every rest point of $X$ corresponds to a rest point of $X'$ of the same index and every closed orbit of index $k$ of $X$ corresponds to two rest points of $X'$, of index $k+1$ and $k$, respectively. We can then apply \Cref{thm:morse-theory-main} to $X'$ in order to get a CW complex.

The CW complex assigned to $X'$ is unique up to cell equivalence, but when replacing $X$ by $X'$ there are some choices involved, which can result in different CW structures. We now present two examples that illustrate this.

\begin{example}\label{ex:minimal-example}
    In \Cref{fig:perturbation1}, we describe the minimal example illustrating the choice when removing a closed orbit. However, in this example, both choices still lead to two cell equivalent CW complexes. The cell equivalence is given by swapping the cells corresponding to $q_1$ and $q_2$. 
\end{example}

\begin{figure}[t]
    \centering
    \begin{subfigure}{0.3\linewidth}
      \includegraphics[width=\linewidth]{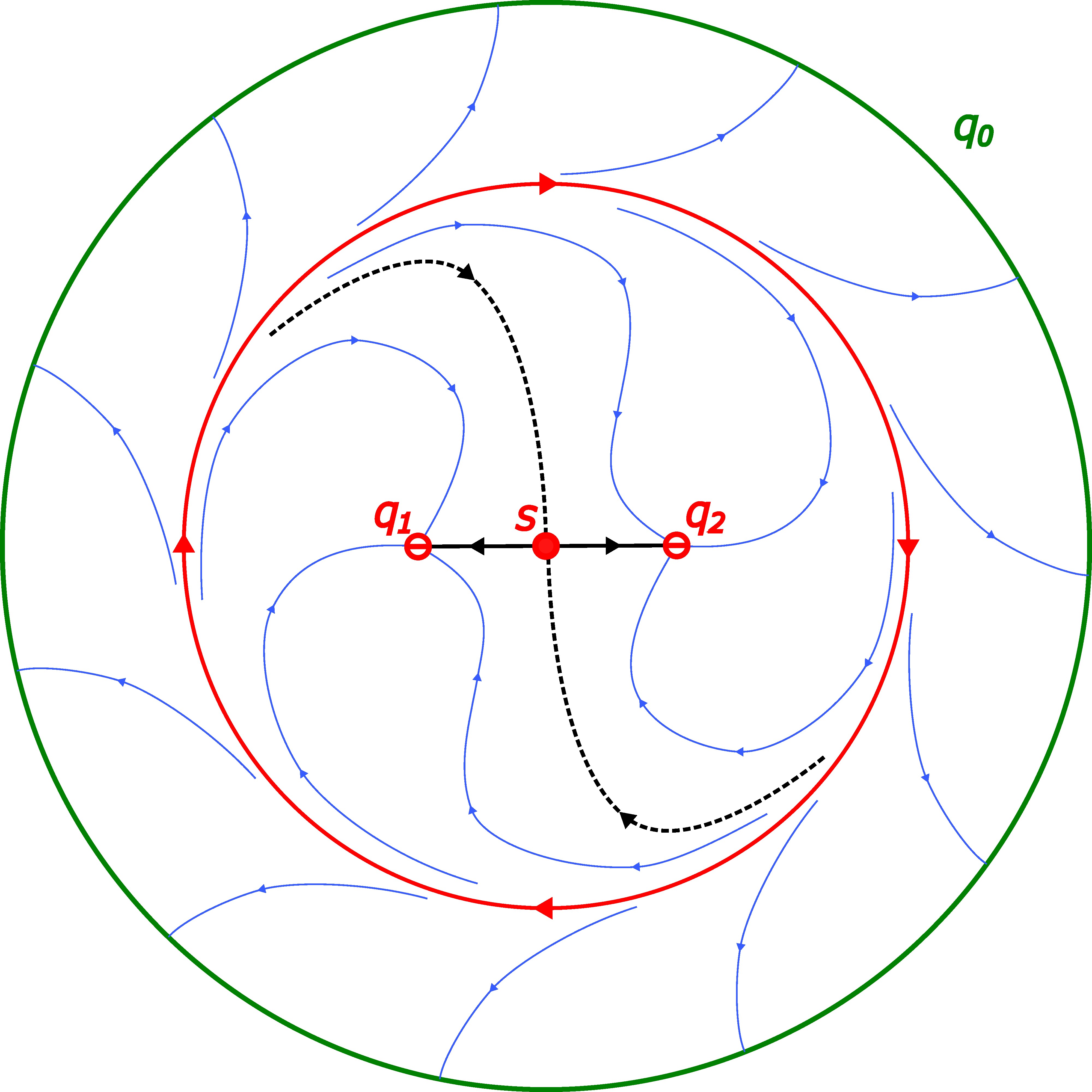}
      \caption{Before the perturbation, the vector field contains a closed orbit.}\label{subfig:pert1-original-with}
    \end{subfigure}
    \hspace{0.02\linewidth}
    \begin{subfigure}{0.31\linewidth}
      \includegraphics[width=\linewidth]{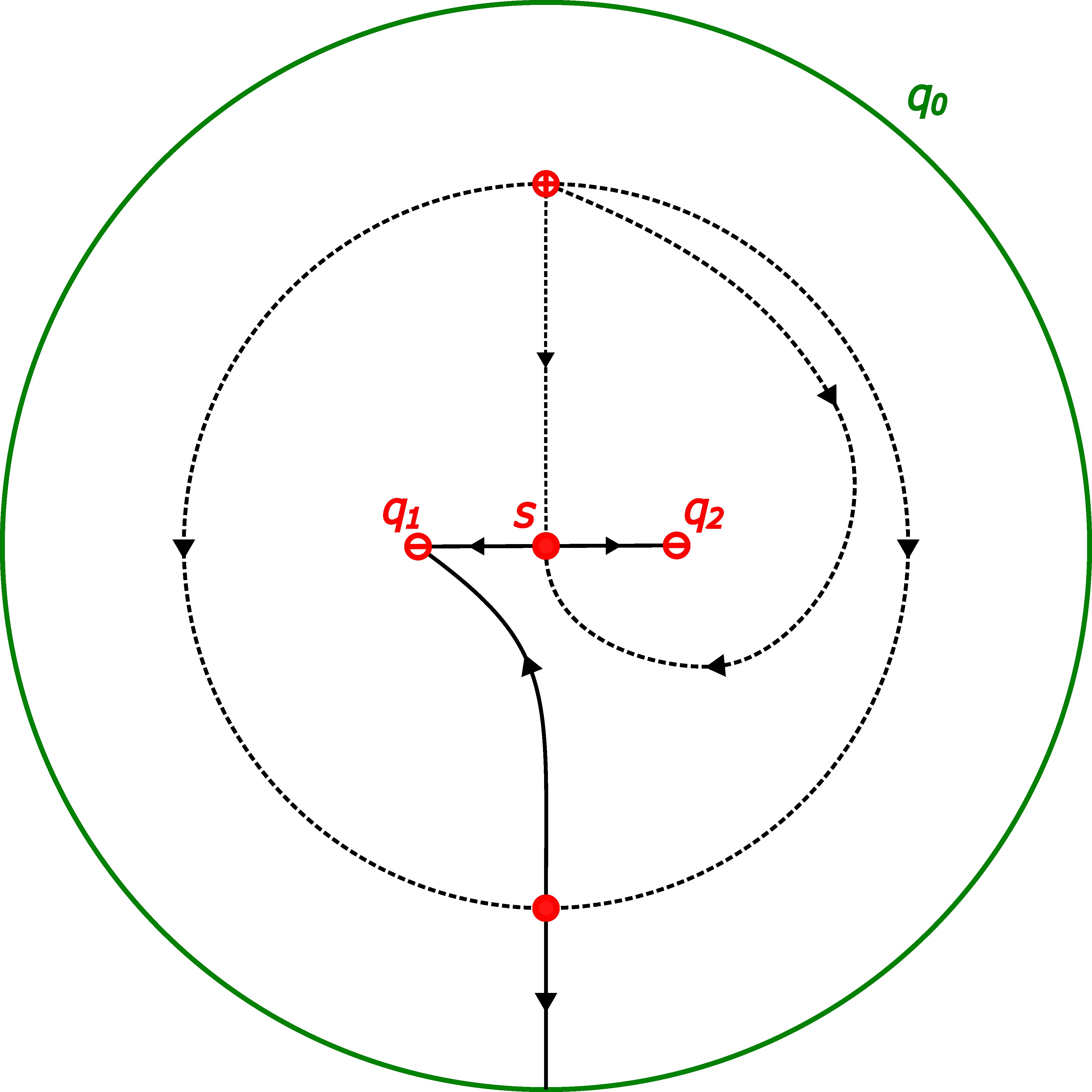}
      \caption{In the first perturbation, the unstable manifold of the new saddle intersects the stable manifold of $q_1$.}\label{subfig:pert1-1}
    \end{subfigure}
    \hspace{0.02\linewidth}
    \begin{subfigure}{0.31\linewidth}
      \includegraphics[width=\linewidth]{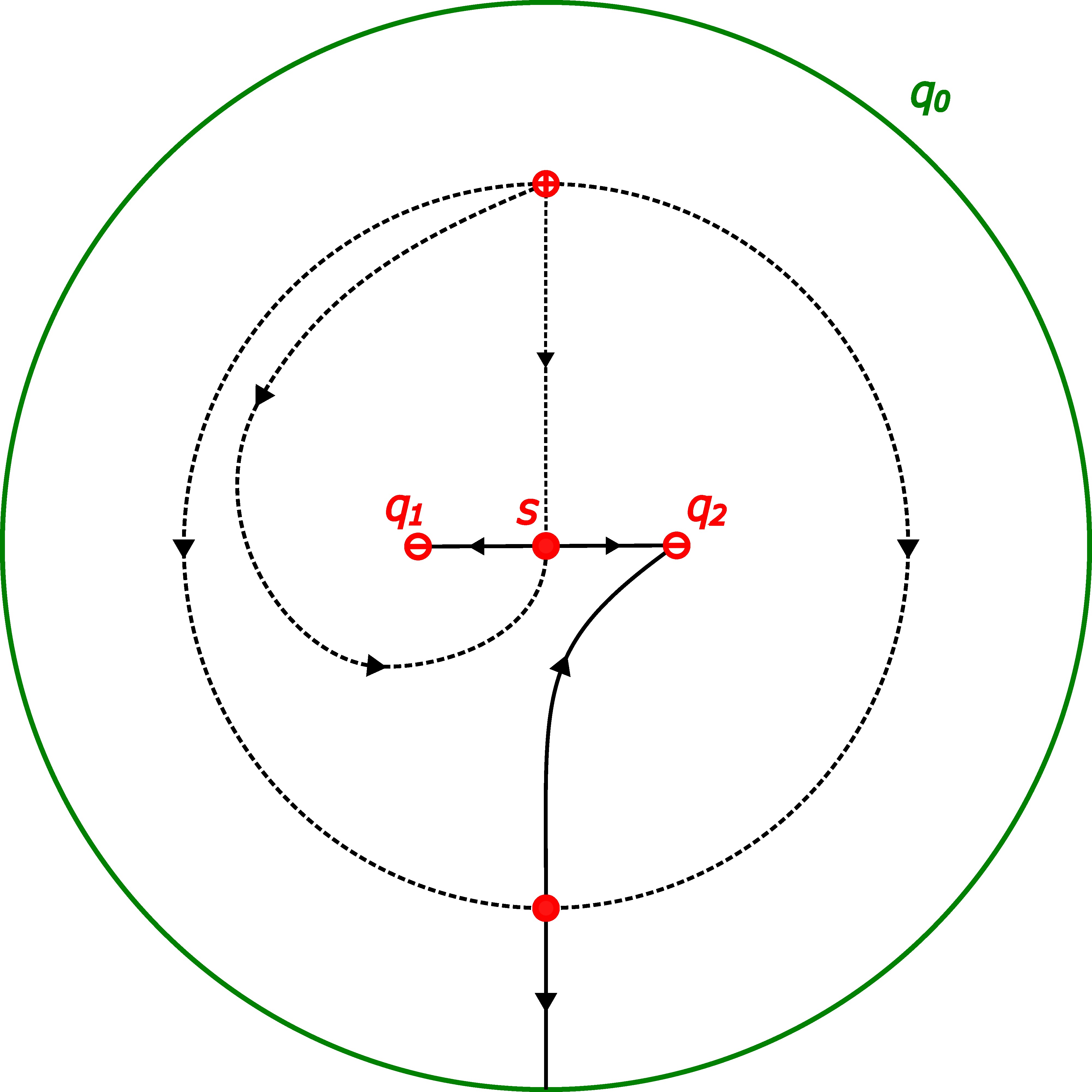}
      \caption{In the second perturbation, the unstable manifold of the new saddle intersects the stable manifold of $q_2$.}\label{subfig:pert1-2}
    \end{subfigure}
    \caption{This vector field has three sinks $q_0$, $q_1$, $q_2$, one saddle $s$ and one repelling orbit $\gamma$. Due to \Cref{prop:replace-closed-orbit} the repeller can be replaced by a source and a saddle. It is not clear a priori if the newly created saddle connects to $q_1$ or $q_2$. Note that in \Cref{subfig:pert1-original-with}, we draw some of additional flow lines in order to visualize the flow of the vector field. In \Cref{subfig:pert1-1} and \Cref{subfig:pert1-2} we omit them.}
    \label{fig:perturbation1}
\end{figure}

We now present a slightly more complicated example, where we end up with two CW complexes that are not cell equivalent. 

\begin{example}\label{ex:not-cell-equiv}
    In the vector field $X$ displayed in \Cref{fig:perturbation2}, we again have one closed orbit, but this time we have more additional rest points. There are three sinks on the inside of the closed orbit, so when we replace it by a source and a saddle, it is not clear a priori to which of these three sinks the flow line starting at the newly introduced saddle should connect to. The first two choices $X_1$ and $X_2$ lead to cell equivalent CW decompositions, but the CW complex resulting from the third perturbation $X_3$ is not cell equivalent to the other two. To see why this is the case, note that for $X_1$ and $X_2$ the number of saddles that each sink is connected to is 1,2,3,4. For $X_3$ however it is 1,3,3,3. This implies that the the posets given by the relation $\le$ on the cells introduced earlier are not isomorphic, hence the CW complexes cannot be cell equivalent. 
\end{example}

\begin{figure}[t]
    \centering
    \begin{subfigure}{0.4\linewidth}
      \includegraphics[width=\linewidth]{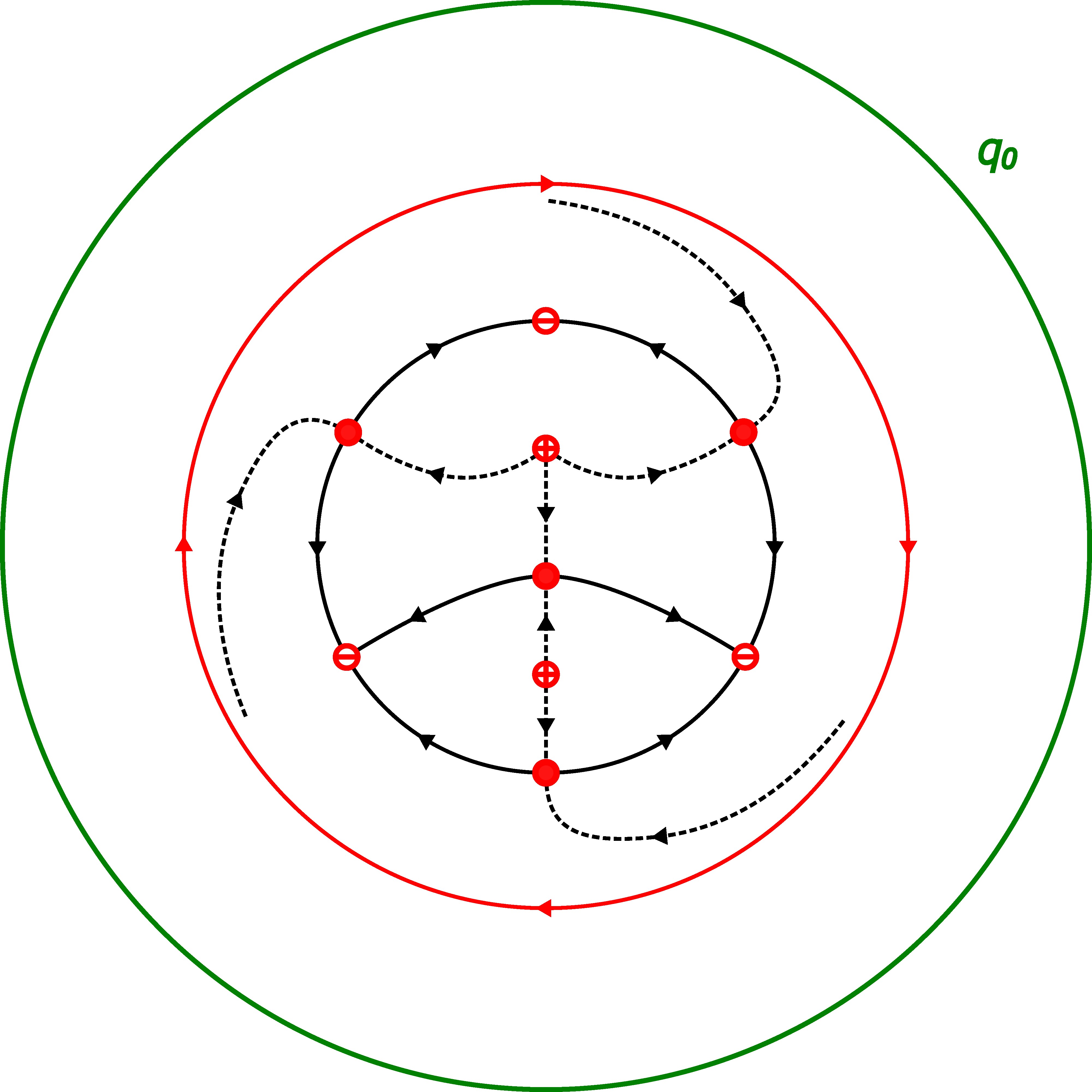}
      \caption{The original vector field $X$}\label{subfig:pert2-a}
    \end{subfigure}
    \hspace{0.02\linewidth}
    \begin{subfigure}{0.4\linewidth}
      \includegraphics[width=\linewidth]{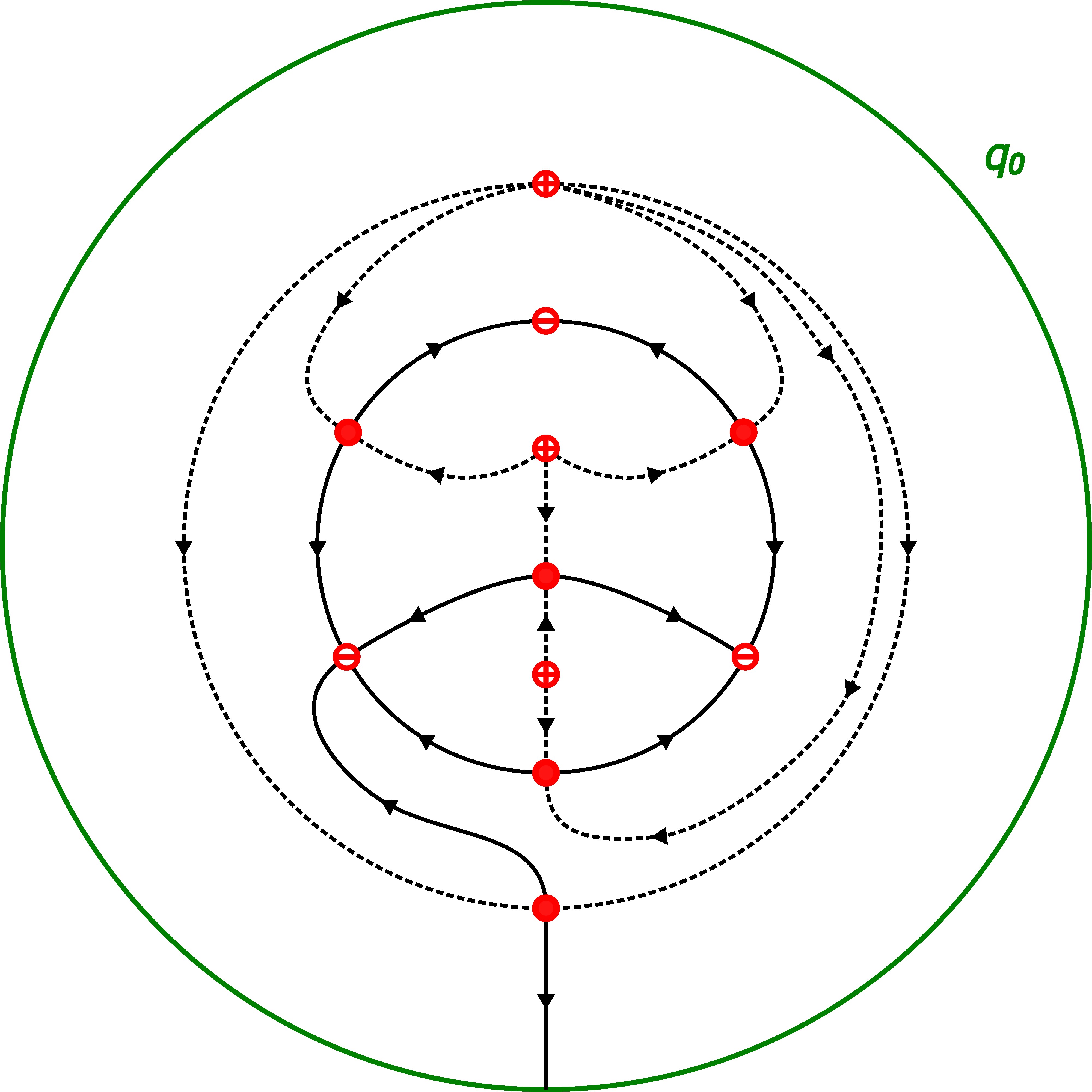}
      \caption{The perturbation $X_1$}\label{subfig:pert2-b}
    \end{subfigure}

    \begin{subfigure}{0.4\linewidth}
      \includegraphics[width=\linewidth]{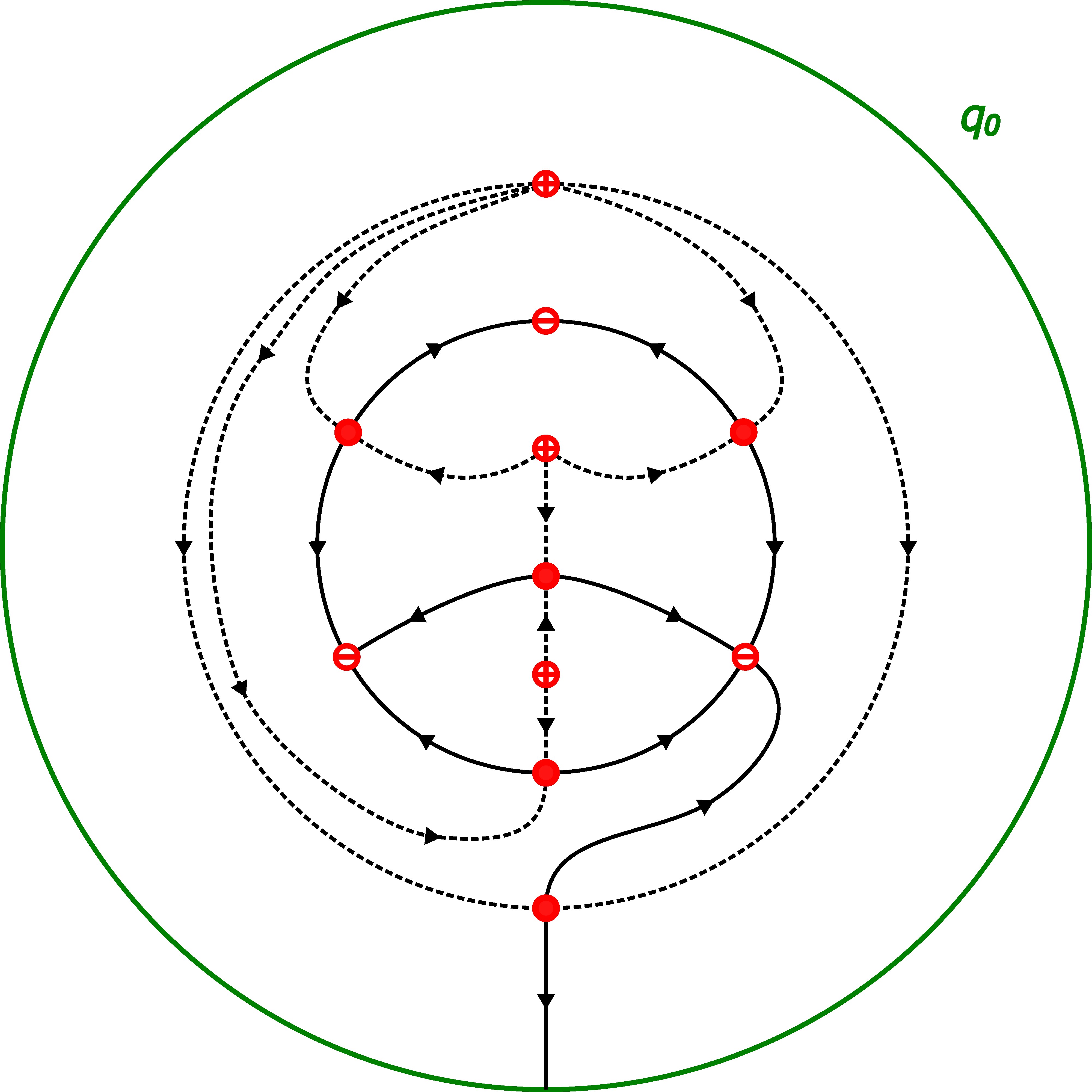}
      \caption{The perturbation $X_2$}\label{subfig:pert2-c}
    \end{subfigure}
    \hspace{0.02\linewidth}
    \begin{subfigure}{0.4\linewidth}
      \includegraphics[width=\linewidth]{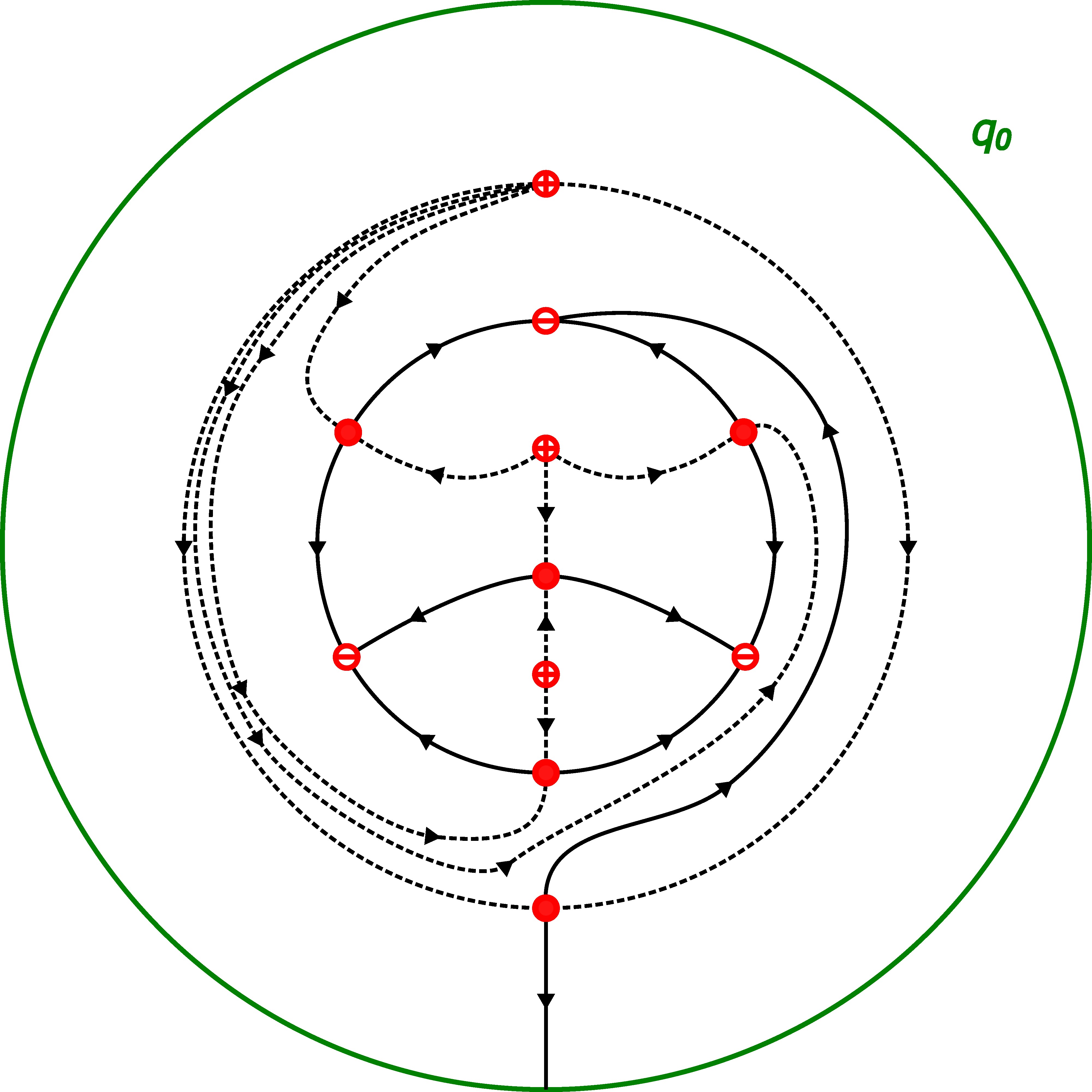}
      \caption{The perturbation $X_3$}\label{subfig:pert2-d}
    \end{subfigure}
    \caption{In this example, the CW complexes assigned to the perturbations $X_1$ and $X_2$ are cell equivalent to each other, but not to the CW complex of $X_3$. We omit drawing any additional flow lines in order not to overload the pictures.}
    \label{fig:perturbation2}
\end{figure}

\section{Non-uniqueness for chain complexes}\label{sec:chain}

In \cite{EidiJost2022}, given a generalized Morse-Smale vector field, Eidi and Jost construct a sequence of vector spaces and linear maps. This is done directly in terms of the critical elements and their intersections of stable and unstable manifolds, without referring to any perturbations. We repeat the main definitions here, for the simpler case of Morse-Smale vector fields (i.e. we exclude homoclinic orbits). In this section we work with coefficients in the field $\Fi:=\Z_2$.

Given a Morse-Smale vector field $X$ and two critical elements $\beta_1$, $\beta_2$ of $X$ (i.e. rest points or closed orbits), we define $\alpha(\beta_1,\beta_2)$ to be the number of connected components (mod 2) of $W^u(\beta_1) \cap W^s(\beta_2)$. 

\begin{definition}
    Given a Morse-Smale vector field $X$ on a closed smooth manifold $M$, the Eidi-Jost complex $(\operatorname{C}^{EJ}_\bullet,\partial^{EJ}_\bullet)$ is defined as follows:
    \begin{itemize}
        \item For any $k\ge 0$,  $\operatorname{C}^{EJ}_k=\operatorname{C}^{EJ}_k(X)$ is the free $\Fi$-vector space generated by all rest points of index $k$, all closed orbits of index $k$ and all closed orbits of index $k-1$. Denote by $\B^{EJ}_k=\B^{EJ}_k(X)$ this generating set (i.e. the canonical basis for $\operatorname{C}^{EJ}_k$).
        \item The linear map $\partial^{EJ}_k=\partial^{EJ}_k(X)\colon \operatorname{C}^{EJ}_k \to \operatorname{C}^{EJ}_{k-1}$ is defined by setting, for all $\beta \in \B^{EJ}_k$,
        \[
        \partial^{EJ}_k(\beta) := \sum_{\beta' \in \B^{EJ}_{k-1}} \alpha(\beta,\beta') \beta' 
        \]
        and extending linearly.
    \end{itemize}
\end{definition}

Note that a closed orbit $\gamma$ of index $k$ contributes one copy to $\B^{EJ}_k$ and $\B^{EJ}_{k+1}$ each. We distinguish the two by adding a minus and a plus sign, i.e. we write $\gamma^- \in \B^{EJ}_k$ and $\gamma^+ \in \B^{EJ}_{k+1}$. 

There is an issue with the proof that $\partial^{EJ}_\bullet$ squares to zero presented in \cite{EidiJost2022}. \Cref{prop:replace-closed-orbit} is used in that proof and it seems as though the uniqueness (which does not hold due to our examples from \Cref{sec:CW}) is assumed implicitly.
We present an alternative proof that works only in the case of two-dimensional manifolds. 

\begin{proposition}
    If $M$ is a two-dimensional closed smooth manifold and $X$ is a Morse-Smale vector field on $M$, then $\partial^{EJ}_1 \circ \partial^{EJ}_2 =0$.
\end{proposition}

\begin{proof}
    Let $M$ be a 2-dimensional manifold and let $X$ be a Morse-Smale vector field on $M$. Write $(C_\bullet,\partial_\bullet)$ and $\B_\bullet$ for the Eidi-Jost complex and the corresponding bases. We want to show that $\partial_1 \circ \partial_2 = 0$. We do an induction over the number $m$ of closed orbits of $X$.

    If $m=0$, then $(C_\bullet,\partial_\bullet)$ is equal to the usual chain complex from Morse homology, for which it is known that the differential squares to zero \cite{BanyagaLecturesOnMorseHomology}.

    We now assume $m>0$ and that we have shown the statement for any Morse-Smale vector field with $m-1$ closed orbits. Let $\gamma$ be a closed orbit of $X$ of index $k$ (where either $k=0$ or $k=1$). We apply \Cref{prop:replace-closed-orbit} to obtain a new Morse-Smale vector field $X'$ which agrees with $X$ outside of a small neighbourhood of $\gamma$ and that instead of the closed orbit $\gamma$ has two rest points $x$ and $y$ of index $k+1$ and $k$, respectively. These two rest points are connected by two flow lines, going from $x$ to $y$. We denote by $(C'_\bullet,\partial'_\bullet)$ and $\B'_\bullet$ the Eidi-Jost complex and the corresponding bases for $X'$. Note that $\B'_{k+1} = (\B_{k+1}\setminus \{\gamma^+\}) \sqcup \{x\}$ and $\B'_k = (\B_k\setminus \{\gamma^-\}) \sqcup \{y\}$. By the induction hypothesis we know that $\partial'_1 \circ \partial'_2 = 0$.

    Let us from now on assume that $k=1$, i.e. $\gamma$ is a repelling orbit. The proof for $k=0$ is analogous and we omit it. For simplicity we identify $\partial_2,\partial_1,\partial'_2,\partial'_1$ with the corresponding representation matrices with respect to the bases $\B_\bullet$ and $\B'_\bullet$, which we assume to be ordered in some way. We claim that the following hold.

    \begin{enumerate}[(i)]
        \item The row corresponding to $\gamma^-$ in $\partial_2$ and the row corresponding to $y$ in $\partial'_2$ are both zero. \label{item:EJ-zero-1}
        \item The matrices $\partial_2$ and $\partial'_2$ are the same. \label{item:EJ-zero-2}
        \item The matrices $\partial_1$ and $\partial'_1$ differ only in the column corresponding to $\gamma^-$ (in $\partial_1$) and $y$ (in $\partial'_1$). \label{item:EJ-zero-3}
    \end{enumerate}

    As for (\ref{item:EJ-zero-1}), note that $\gamma$ is a repeller, so $\alpha(\beta,\gamma)=0$ for all other critical elements $\beta$. Thus the row corresponding to $\gamma^-$ in $\partial_2$ is zero. For $\partial'_2$ on the other hand, there are two flow lines from $x$ to $y$ and no other flow lines from any other critical element to $y$ (because $y$ is a saddle point). These two flow lines cancel out due to $\Z_2$ coefficients and thus also the row corresponding to $y$ in $\partial'_2$ is zero.

    For (\ref{item:EJ-zero-2}) and (\ref{item:EJ-zero-3}), note that all remaining entries are defined by counting intersections of stable and unstable manifolds of critical elements different from $\gamma$. Assuming that we have chosen small enough the neighbourhood of $\gamma$ on which $X$ and $X'$ differ, these intersection numbers are the same for $X$ and $X'$, hence the according entries in the boundary matrices are the same.

    From these three claims it follows now that $\partial_1 \circ \partial_2 = \partial'_1 \circ \partial'_2$, since the only difference between the two is in one column of $\partial_1$ (resp. $\partial'_1$) that gets multiplied by a zero row in $\partial_2$ (resp. $\partial'_2$) anyway. Hence, as we knew already that $\partial'_1 \circ \partial'_2=0$, it follows that $\partial_1 \circ \partial_2=0$.
\end{proof}

\begin{figure}[ht]
    \centering
    \includegraphics[width=8cm]{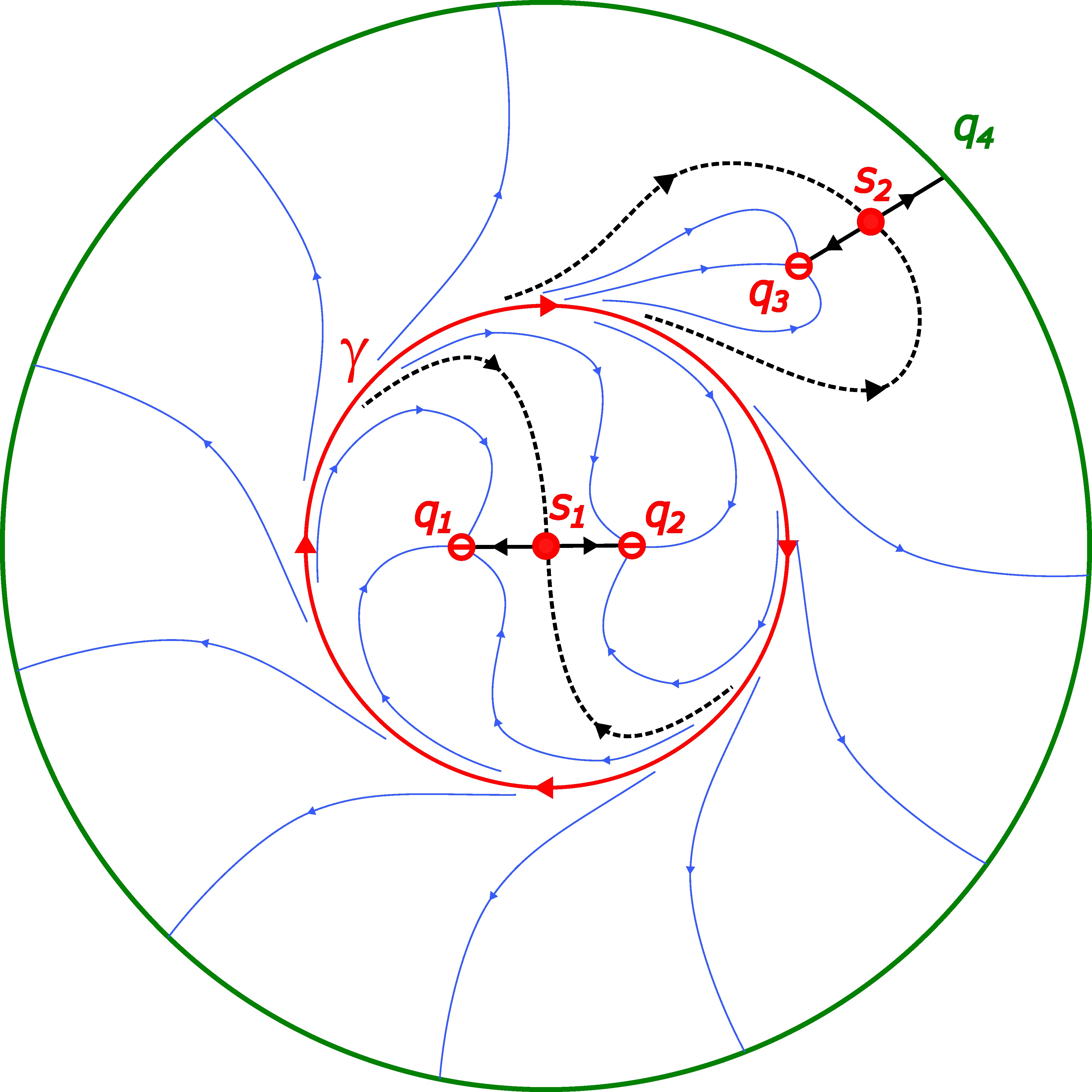}
    \caption{This is an example of a Morse-Smale vector field on the 2-sphere, where the homology of the Eidi-Jost complex differs from the singular homology of $S^2$.}
    \label{fig:closed-orbit-1}
\end{figure}

Now we show by an example that the homology of the Eidi-Jost complex of a Morse-Smale vector field on a surface can be different from the singular homology of the surface, implying that the proof of \cite[Theorem 2.5]{EidiJost2022} is flawed (this has been acknowledged by the authors of \cite{EidiJost2022} by private communication).

\begin{example}
    For the vector field depicted in \Cref{fig:closed-orbit-1}, the associated Eidi-Jost complex looks as follows.
        \begin{center}
        \begin{tikzcd}[ampersand replacement=\&]
            \Fi= \Span(\gamma^+) \ar[rr,"\partial^{EJ}_2 = \begin{bmatrix}
                0\\
                0\\
                0
            \end{bmatrix}"]
            \&\&
            \Fi^3=\Span(s_1,s_2,\gamma^-) \ar[rr,"\hspace{-15pt}\partial^{EJ}_1 = {\begin{bmatrix}
                1 &0 &1\\
                1 &0 &1\\
                0 &1 &1\\
                0 &1 &1
            \end{bmatrix}}"]
            \&\&\Fi^4=\Span(q_1,q_2,q_3,q_4). 
        \end{tikzcd}
        \end{center} 
    Indeed, since $\partial^{EJ}_1$ in this example has rank $2$, it follows that we get a one-dimensional homology group in degree $1$ and two-dimensional homology in degree $0$. This does not agree with the homology of the $2$-sphere.
\end{example}

\begin{figure}[t]
    \centering
    \begin{subfigure}{0.45\linewidth}
      \includegraphics[width=\linewidth]{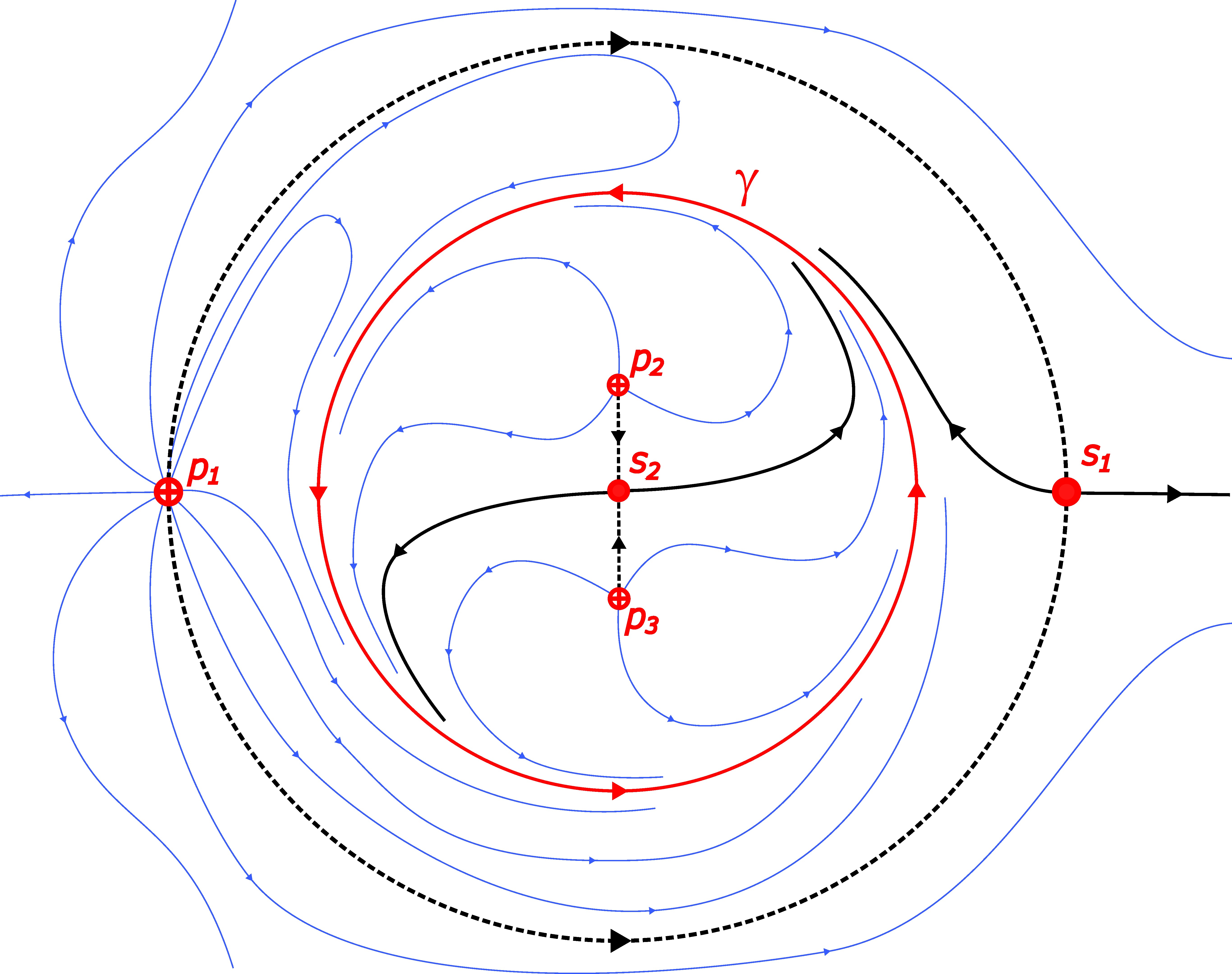}
      \caption{The Morse-Smale vector field $X$ on $\R^2$}\label{subfig:MSvf-R2}
    \end{subfigure}
    \hspace{0.05\linewidth}
    \begin{subfigure}{0.45\linewidth}
      \includegraphics[width=\linewidth]{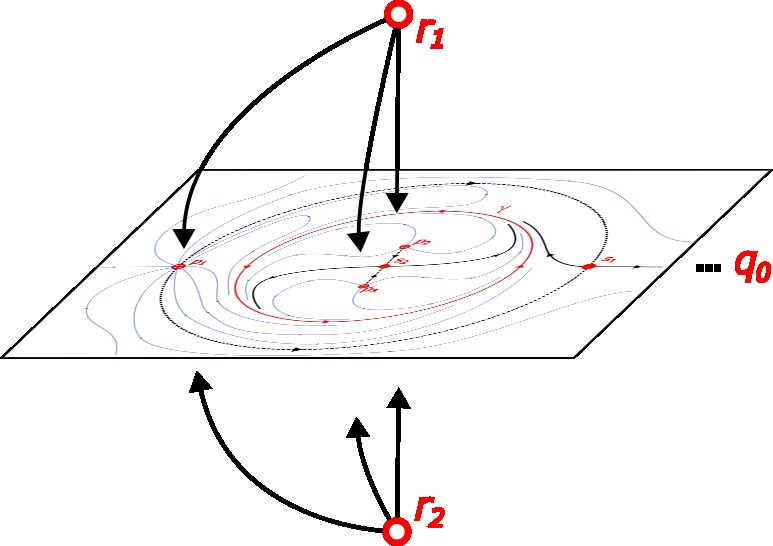}
      \caption{The extension $w$ is a vector field on $\R^3$ with a sink at infinity.}\label{subfig:MSvf-R3}
    \end{subfigure}
    \caption{This is an example of a Morse-Smale vector field on $S^3$ for which the Eidi-Jost differential does not square to zero.}
    \label{fig:3D-vector-field}
\end{figure}

Moreover, in the following example, the Eidi-Jost differential does not square to zero.

\begin{example}
    We build a $3$-dimensional Morse-Smale vector field $V$ as follows. Consider first the vector field $X$ on $\R^2$ from \Cref{subfig:MSvf-R2}. It has sources $p_1,p_2,p_3$, saddles $s_1,s_2$, and a closed orbit $\gamma$ of index $0$. Outside of what is shown in the image, the vector field continues to the rest of $\R^2$ without any further critical elements, all of the flow lines going towards infinity. More precisely, we may assume that the point $s_2$ lies at the origin and that for $\norm{\Vec{x}}$ large enough, $X(\Vec{x})=\Vec{x}$. 

    We now describe how to extend $X=(X_1,X_2)$ to a vector field $V$ on $\R^3$. The idea is that on the plane $\R^2 \times \{0\} \subseteq \R^3$, $V$ agrees with $X$ and outside of this plane, in a small neighbourhood of $\R^2\times \{0\}$ everything flows toward that plane, and at the points $(0,0,1)$ and $(0,0,-1)$ we have two rest points of index $3$, see \Cref{subfig:MSvf-R3}. More explicitly, we choose a value $0< \varepsilon < \frac{1}{2}$ and demand that $V$ has the following properties:
    \[
    V(x,y,z) = \begin{cases}
        (X_1(x,y),X_2(x,y),-z) &\text{if } |z| < \varepsilon \\
        (x,y,z-1) &\text{if } z > \frac{3}{4} \\
        (x,y,z+1) &\text{if } z < -\frac{3}{4}. \\
    \end{cases}
    \]
    On the slices $\R^2 \times [\varepsilon,\frac{3}{4}]$ and $\R^2 \times [-\frac{3}{4},-\varepsilon]$ we interpolate somehow, the exact way we do this is not important as long as we introduce no additional critical elements. All the critical elements of $X$ become critical elements of $V$. One can check that they are still hyperbolic and that their unstable manifolds stay the same, while their stable manifolds get one extra dimension, locally perpendicular to the plane $\R^2\times \{0\}$. This means that they keep the same index. Additionally, $V$ has the rest points $r_1$ and $r_2$ of index $3$. The vector field $V$ is thus again a Morse-Smale vector field.

    We can rescale $V$ so that $\norm{V(\Vec{x})} \to 0$ when $\norm{\Vec{x}}\to \infty$, which allows us to add a rest point for $V$ of index $0$, call it $q_0$, at infinity and thus view $V$ as a vector field on $S^3$. In order to write down the Eidi-Jost complex, note that $\alpha(r_i,p_j)=1$ for $i=1,2$ and $j=1,2,3$. The other relevant values of $\alpha$ can be read off \Cref{subfig:MSvf-R2}. The resulting sequence is
    \begin{center}
        \begin{tikzcd}[ampersand replacement=\&]
            \Span(r_1,r_2) \ar[rr,"\hspace{-15pt}\partial^{EJ}_3 = {\begin{bmatrix}
                1 &1\\
                1 &1\\
                1 &1
            \end{bmatrix}}"]
            \&\&
            \Span(p_1,p_2,p_3) \ar[rr,"\hspace{-15pt}\partial^{EJ}_2 = {\begin{bmatrix}
                0 &0 &0\\
                0 &1 &1\\
                1 &1 &1
            \end{bmatrix}}"]
            \&\&
            \Span(s_1,s_2,\gamma^+) \ar[rr,"\hspace{-15pt}\partial^{EJ}_1 = {\begin{bmatrix}
                1 &0 &0\\
                1 &0 &0
            \end{bmatrix}}"]
            \&\&\Span(q_0,\gamma^-)
        \end{tikzcd}
    \end{center} 
    and we see that $\partial^{EJ}_2 \circ \partial^{EJ}_3 \neq 0$.
\end{example}

In conclusion, we have highlighted uniqueness issues arising when removing closed orbits from Morse-Smale vector field. This resonates with the results of Reineck about the non-uniqueness of the connection matrix \cite{ReineckConnectionMatrix1990}.

\paragraph{Acknowledgements.} I would like to thank my supervisor Claudia Landi for discussing this topic many times with me and giving helpful suggestions, John Franks for taking my doubts about his article into consideration, even though it was published over 40 years ago, and Marzieh Eidi and Jürgen Jost for their availability and effort to discuss their results.

\bibliographystyle{amsplain}
\bibliography{Bibliography}

\end{document}